\documentclass[dvipsnames,11pt,reqno,twoside,final]{amsart}
\usepackage[a4paper,left=3cm,right=3cm,top=3cm,bottom=3cm]{geometry}
\usepackage[T1]{fontenc}
\usepackage[utf8]{inputenc}
\usepackage{xcolor}
\usepackage{amssymb,mathtools}  
\usepackage{dsfont}
\usepackage{enumitem}
\usepackage{subcaption}
\usepackage{booktabs}

\usepackage{preamble/mhequ}

\usepackage{amsthm}

\newcounter{counterEnvDefault}
\numberwithin{counterEnvDefault}{section}

\theoremstyle{plain}

\newtheorem{lemma}[counterEnvDefault]{Lemma}
\newtheorem*{lemma*}{Lemma}

\newtheorem{theorem}[counterEnvDefault]{Theorem}
\newtheorem*{theorem*}{Theorem}

\newtheorem{proposition}[counterEnvDefault]{Proposition}
\newtheorem*{proposition*}{Proposition}

\newtheorem*{corollary*}{Corollary}

\newtheorem*{assumption*}{Assumption}

\theoremstyle{definition}

\newtheorem*{exercise*}{Exercise}

\newtheorem{definition}[counterEnvDefault]{Definition}
\newtheorem*{definition*}{Definition}

\newtheorem*{notation*}{Notation}

\newtheorem{remark}[counterEnvDefault]{Remark}
\newtheorem*{remark*}{Remark}

\newtheorem*{claim*}{Claim}

\newtheorem*{assertion*}{Assertion}

\usepackage{microtype}
\renewcommand{\phi}{\varphi}
\renewcommand{\epsilon}{\varepsilon}

\usepackage{hyperref}
\definecolor{colorlinks}{RGB}{0, 24, 168}
\definecolor{colorcites}{RGB}{124, 10, 2}
\hypersetup{
    colorlinks=true,
    linkcolor=colorlinks,
    citecolor=colorcites,
    urlcolor=colorlinks,
    pdfborder={0 0 0}
}

\newcommand\blank{\,\cdot\,}

\newcommand\XY{{\operatorname{XY}}}
\newcommand\Height{{\operatorname{Height}}}

\newcommand\ind[1]{\mathds{1}_{#1}}
\newcommand\true[1]{\mathds{1}[{#1}]}

\newcommand\diff{{\mathrm{d}}}

\newcommand\C{\mathbb C}

\renewcommand\H{\mathbb H}

\newcommand\R{\mathbb R}
\renewcommand\S{\mathbb S}

\newcommand\Z{\mathbb Z}

\newcommand\calA{\mathcal A}

\newcommand\calC{\mathcal C}

\newcommand\calE{\mathcal E}

\newcommand\calP{\mathcal P}

\newcommand\calS{\mathcal S}

\newcommand\frakm{\mathfrak m}

\usepackage[
    style=alphabetic,
    date=year,
    maxnames=3,
    maxbibnames=99
]{biblatex}
\AtEveryBibitem{%
    \clearfield{url}%
    \clearfield{urldate}%
    \clearfield{urlday}%
    \clearfield{urlmonth}%
    \clearfield{urlyear}%
    \clearfield{doi}%
    \clearfield{isbn}%
    \clearfield{issn}%
}

\makeatletter
\@namedef{subjclassname@2020}{\textup{2020} Mathematics Subject Classification}
\makeatother

\title{The BKT transition and surface tension differentiability}
\subjclass[2020]{Primary 82B20, 82B26, 82B41}

\author{Thibault Durand}
\address{ENS Ulm, Sorbonne Université, LPSM}
\email{thibault.durand@ens.psl.eu}

\author{Piet Lammers}
\address{CNRS, Sorbonne Université, LPSM}
\email{piet.lammers@cnrs.fr}

\date{27 May 2026}
\keywords{Berezinskii--Kosterlitz--Thouless transition, XY model, height functions, mass,
correlation length, surface tension}

\begin{document}
\newcommand\BC{\mathbf{C}}
\newcommand\BP{\mathbf{P}}
\newcommand\BM{\mathbf{M}}
\newcommand\crev{c_{\mathrm{rev}}}
\newcommand\cfav{c_{\mathrm{fav}}}

\begin{abstract}
    Under the duality between the two-dimensional XY model and an integer-valued height
    function, the BKT transition is expected to correspond to the disappearance of a
    corner in the surface tension; in the delocalised phase, its zero-slope curvature
    should determine the Gaussian free field prefactor.
    We prove that the XY mass equals the right derivative at zero slope of the dual
    height function's free energy, with a uniform quadratic error bound near the origin.
    Thus the massive phase is exactly the corner regime, while in the BKT phase the
    surface tension is quadratically bounded at zero slope.
    The proof combines Kadanoff--Ceva duality, Ginibre's inequality, and a pushing lemma
    (an estimate of
    Russo--Seymour--Welsh type) for the cable height function.
\end{abstract}

\maketitle

\setcounter{tocdepth}{1}
\tableofcontents

\section{Main results and proof strategy}

\subsection{Formal statement of the main result}

\begin{definition}[XY model]
    \label{def:xy_model}
    Consider a finite simple graph $G=(V,E)$ with edge weights
    $J=(J_{uv})_{uv\in E}\subset[0,\infty)$.
    Let $\Omega=\S^V$, where $\S\subset\C$ denotes the unit circle, and endow
    $\Omega$ with the product Haar probability measure $\diff\sigma$.
    The \emph{XY model} on $G$ with \emph{coupling constants} $J$ is the Gibbs measure
    $\langle\blank\rangle_{G,J}$ on $\Omega$ defined by the formula
    \begin{equation}
        \label{eq:XY_definition}
        \langle X\rangle_{G,J}:=\frac1{Z_{G,J}}\int_\Omega X(\sigma)
        e^{\frac12\sum_{uv\in\vec E}J_{uv}\sigma_u\bar\sigma_v} \diff\sigma,
    \end{equation}
    where $\vec E$ is the set of directed edges, and $Z_{G,J}\in(0,\infty)$ is the
    \emph{partition function}.
    Often $J$ is taken to be a constant $\beta\in[0,\infty)$, in which case we write
    $\langle\blank\rangle_{G,\beta}$.
    The parameter $\beta$ is called the \emph{inverse temperature}.

    We are interested in the \emph{mass} (the inverse of the \emph{correlation length})
    of the XY model in 2D.
    Write $\Lambda_n$ for the nearest-neighbour graph on $[-n,n]^2\cap\Z^2$.
    Then the \emph{mass} $\frakm_\XY(\beta)$ is defined via
    \begin{equation}
        \label{eq:XY_mass_definition}
        \frakm_\XY(\beta):=\lim_{k\to\infty}-\frac1k\log\lim_{n\to\infty}\langle\sigma_0\bar\sigma_{k
        e_1}\rangle_{\Lambda_n,\beta}\in[0,\infty].
    \end{equation}
\end{definition}

For an introduction to the XY model,
see~\cite{PeledSpinka_2019_LecturesSpinLoop,FriedliVelenik_2017_StatisticalMechanicsLattice}.
The mass of the 2D XY model characterises the exponential decay rate of the two-point
correlation function.
By the Ginibre inequality~\cite{Ginibre_1970_GeneralFormulationGriffiths}, it is a
non-increasing function of $\beta$.
The \emph{Berezinskii--Kosterlitz--Thouless (BKT) transition point}
$\beta_{\mathrm{BKT}}$ is the threshold at which this mass vanishes and the decay of the
two-point function changes from exponential to polynomial.
Fröhlich and
Spencer~\cite{FrohlichSpencer_1981_KosterlitzThoulessTransitionTwodimensional} proved
polynomial decay for sufficiently large $\beta$, showing in particular that
$\beta_{\mathrm{BKT}}\in(0,\infty)$.
See Section~\ref{section:background} for more context.

Recent work has analysed the BKT transition from several complementary
perspectives~\cite{EngelenburgLis_2023_ElementaryProofPhase,AizenmanHarelPeled_2022_DepinningIntegerrestrictedGaussian,GarbanSepulveda_2023_QuantitativeBoundsVortex,Lammers_2023_BijectingBKTTransition}.
Several of these works use the dual height function, which we now introduce.
It is defined on the \emph{dual graph} of $\Lambda_n$,
namely the graph $\Lambda_n^*$ whose vertices are centres of $1\times 1$ squares of $\Z^2$
that share at least one edge with $[-n,n]^2$,
and whose edges connect nearest-neighbour vertices.
Its \emph{boundary} is defined as $\partial\Lambda_n^*:=\Lambda_n^*\setminus[-n,n]^2$.
See Figure~\ref{fig:dual_graph}, \textsc{Left} for an example.

\begin{figure}
    \centering
    \includegraphics{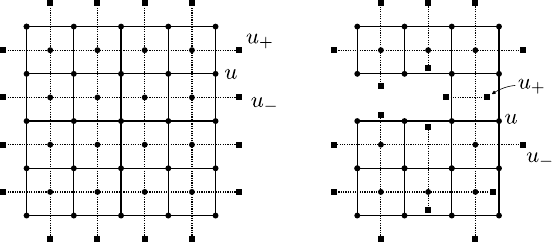}
    \caption{\textsc{Left}: The primal graph $\Lambda_2$ (solid lines) and its dual graph
        $\Lambda_2^*$ (dashed lines).
        The vertices in $\partial\Lambda_2^*$ are marked with squares rather than circles.
        \textsc{Right}: A simply connected continuum domain $\Gamma$ (defined in
        Section~\ref{sec:the_cable_system})
        and its associated primal graph.
        The edge lengths in $\Gamma$ do not affect the primal graph, but they do affect the
    weights $J$ on the edges of the primal graph.}
    \label{fig:dual_graph}
\end{figure}

\begin{definition}[Height function]\label{def:dual_height_function}
    Consider a finite simple graph $G=(V,E)$ with edge weights
    $J=(J_{uv})_{uv\in E}\subset[0,\infty)$.
    Let $\emptyset\neq\partial G\subset V$ be a distinguished subset of vertices,
    called the \emph{boundary}, and let $\zeta\in\Z^{\partial G}$ be a
    \emph{boundary height function}.
    The \emph{height function} on $G$ with coupling constants $J$ and boundary conditions
    $\zeta$ is the Gibbs measure $\mu_{G,J}^\zeta$ on $\Z^V$ defined by the formula
    \begin{multline}
        \label{eq:height_function_definition}
        \mu_{G,J}^\zeta[\{h\}] := \frac1{Z_{G,J}^\zeta}
        \true{h|_{\partial G}=\zeta} \prod_{uv\in E} w_{J_{uv}}(h_u-h_v);
        \\
        w_{J_{uv}}(k) := \sum_{i,j\in\Z_{\geq 0},\,i-j=k} \frac{(J_{uv}/2)^i}{i!}
        \frac{(J_{uv}/2)^j}{j!}.
    \end{multline}
    Here $Z_{G,J}^\zeta$ is the \emph{partition function}.
    We write $\mu_{G,\beta}^\zeta$ when $J\equiv\beta\in[0,\infty)$ is constant.
\end{definition}

We are particularly interested in the family $(\mu_{\Lambda_n^*,\beta}^0)_n$.
Let $F_u$ denote the face to the north-east of some $u\in\Z^2$,
and define
\begin{equation}
    M_n:\Z^2\times\Z^2\to[0,\infty),\qquad
    (u,v)\mapsto\mu_{\Lambda_n^*,\beta}^0[h_{F_u}h_{F_v}],
\end{equation}
with the convention that $M_n(u,v)=0$ when one of the faces is outside $\Lambda_n^*$.
The \emph{mass} of the height function is defined as
\begin{equation}
    \label{eq:Heights_mass_definition}
    \frakm_\Height(\beta):=\lim_{k\to\infty}-\frac1k\log\lim_{n\to\infty}
    (1\wedge M_n(0,ke_1))
    \in[0,\infty].
\end{equation}

The relationship between the two models is partially understood. We shall use the
following known connections as motivation and context.
\begin{itemize}
    \item If $\frakm_\XY(\beta)=0$, then the two-point function has a polynomial lower
        bound; together with the McBryan--Spencer upper bound this gives polynomial
        decay~\cite{McBryanSpencer_1977_DecayCorrelationsSOnsymmetric,EngelenburgLis_2023_ElementaryProofPhase}.
    \item If $\frakm_\Height(\beta)=0$, then $M_n(0,0)$ grows at least logarithmically in
        $n$~\cite{AizenmanHarelPeled_2022_DepinningIntegerrestrictedGaussian,Lammers_2023_DichotomyTheoryHeight}.
    \item For any $\beta\in[0,\infty)$, we have
        $2\frakm_\XY(\beta)=\frakm_\Height(\beta)$~\cite{Lammers_2023_BijectingBKTTransition}.
\end{itemize}
The purpose of the present work is to connect these quantities with a third one:
the \emph{free energy} of the height function model with tilted boundary conditions.
In the definition below, we identify any $s\in\R$ with the boundary condition on
$\partial\Lambda_n^*$
defined via
\begin{equation}
    \partial\Lambda_n^*\to\Z,\,x\mapsto \lfloor s x_1\rfloor.
\end{equation}
This parameter is called the \emph{slope}.

\begin{definition}[Free energy]
    \label{def:free_energy}
    For fixed $\beta\in(0,\infty)$, the free energy is a function $f_\beta:\R\to\R$ defined by
    \begin{equation}
        \label{eq:free_energy}
        f_\beta(s) := \lim_{n\to\infty}-\frac{1}{|\Lambda_n^*|}\log Z_{\Lambda_n^*,\beta}^s.
    \end{equation}
    With this normalisation, $f_\beta$ is symmetric and strictly
    convex~\cite{Sheffield_2005_RandomSurfaces}.
    We write $f'_\beta$ for its \emph{right} derivative at points where the distinction
    matters, and $f''_\beta$ for its second derivative when it exists.
    If $f'_\beta(0)>0$, then the free energy is not differentiable at $0$ and we say that
    the free energy has a \emph{corner} at $0$.
\end{definition}

\begin{theorem}[The XY model's mass as the free energy's derivative]
    \label{thm:Main}
    Consider the two-dimensional XY and height models introduced above.
    There exists a universal constant $c>0$ such that, for any
    $\beta\in(0,\infty)$ and $s\in(0,1/1000)$,
    \begin{equation}
        \label{eq:main}
        s\cdot\frakm_\XY(\beta)\leq
        f_\beta(s)-f_\beta(0)\leq s\cdot\frakm_\XY(\beta) + c s^2/2.
    \end{equation}

    In particular, the theorem gives the following consequences.
    \begin{itemize}
        \item We have $2\frakm_\XY(\beta)=\frakm_\Height(\beta)=2f'_\beta(0)$ for all
            $\beta\in(0,\infty)$.
        \item If $\frakm_\XY(\beta)=0$, then for small $s>0$ we have
            \begin{equation}
                f_\beta(s)-f_\beta(0)\leq c s^2/2.
            \end{equation}
            If $f''_\beta(0)$ is well-defined, then $f''_\beta(0)\leq c$.
    \end{itemize}
\end{theorem}

These results are reminiscent of analogous statements for the six-vertex model
(a height function model)
and the random-cluster model (a percolation model).
The two models are linked via the Baxter--Kelland--Wu
coupling~\cite{BaxterKellandWu_1976_EquivalencePottsModel}.
In the six-vertex model, derivatives of the free energy can be analysed via the Bethe Ansatz.
In~\cite{Duminil-CopinKozlowskiKrachun_2022_SixvertexModelsFree},
it is proved that the free energy is twice differentiable at slope $0$ or has a corner at
slope $0$ depending on the parameters.
In the former case, the second derivative can be computed explicitly,
which determines (in part of the phase diagram) the scaling of the Gaussian free field in
the scaling limit,
and certain critical exponents in the associated random-cluster model,
see~\cite{Duminil-CopinKajetanKozlowskiLammers_2026_GaussianFreeField} and the references therein.
When the free energy has a corner, the mass of the random-cluster model
(associated with the right derivative of the free energy) is positive and calculated
explicitly in~\cite{Duminil-CopinGagnebinHarel_2021_DiscontinuityPhaseTransition}.
In that integrable setting, these results give a more explicit picture of the relation
between the free energy and the associated mass.

\subsection{Proof strategy}

\subsubsection*{Overview of the tools}

It has been known since
Kramers--Wannier~\cite{KramersWannier_1941_StatisticsTwodimensionalFerromagnet_I,KramersWannier_1941_StatisticsTwodimensionalFerromagnet_II}
that
the planar Ising model has a dual model that is also an Ising model.
For example, the models share the same partition function (up to trivial correction factors).
They observed that the dual temperature is the inverse of the primal temperature,
which suggests that the critical point coincides with the self-dual temperature.
Kadanoff--Ceva~\cite{KadanoffCeva_1971_DeterminationOperatorAlgebra} argued
that the correspondence extends from partition functions to certain correlation functions.
More precisely, inserting correlation functions on one side of the duality corresponds to
perturbing the partition function on the other side of the duality via so-called
\emph{disorder operators}.
In our very specific context, the correlation functions measure spins at vertices on the
outer boundary of the planar graph (or, equivalently, on the vertices of a single face,
designated as the \emph{outer face}).
In that case, the disorder operators can be interpreted as a change in the boundary
condition of the height function on the dual graph.
This boundary interpretation lets us use tools on both sides of the duality:
\begin{itemize}
    \item On the XY side, we have the Ginibre
        inequality~\cite{Ginibre_1970_GeneralFormulationGriffiths} and the
        Brydges--Fröhlich--Spencer (BFS)
        representation~\cite{BrydgesFrohlichSpencer_1982_RandomWalkRepresentation},
    \item On the height function side, the system can be analysed via a percolation
        representation of the height
        function~\cite{Sheffield_2005_RandomSurfaces,Lammers_2023_DichotomyTheoryHeight}
        using the Fortuin--Kasteleyn--Ginibre (FKG)
        inequality~\cite{FortuinKasteleynGinibre_1971_CorrelationInequalitiesPartially}
        for the absolute value of the height
        function~\cite{LammersOtt_2024_DelocalisationAbsolutevalueFKGSolidonsolid},
        and in particular the so-called \emph{pushing
        lemma}~\cite{Duminil-CopinSidoraviciusTassion_2017_ContinuityPhaseTransition,Lammers_2023_DichotomyTheoryHeight}.
\end{itemize}

\subsubsection*{Global organisation}

The proof is organised as follows.
Section~\ref{sec:kadanoff_ceva} recalls the Kadanoff--Ceva correspondence in the
generality needed here.
Section~\ref{sec:lower_bound} proves the lower bound in Equation~\eqref{eq:main}.
The argument uses the FKG inequality and the Kadanoff--Ceva correspondence to
bound a ratio of partition functions
in terms of a product over two-point correlation functions, and then uses the Ginibre
inequality to relate this to the mass of the XY model.
Sections~\ref{sec:upper_bound_overview}--\ref{sec:the_technical_step} are devoted to the
upper bound in Equation~\eqref{eq:main}, which requires more work.

Section~\ref{sec:upper_bound_overview} contains an overview of the proof.
It contains two gaps that are filled in in the subsequent sections.
The first gap can be thought of as a reverse Ginibre inequality (stated below).
This is the main input for the upper bound.
The second gap is a technical step;
its proof is similar to the proof of the reverse Ginibre inequality, but
its role in the argument is narrower.
Section~\ref{sec:the_cable_system} introduces the \emph{cable system},
which provides a framework for the proof of Sections~\ref{sec:reverseinequality}
and~\ref{sec:the_technical_step}.
Proofs of the reverse Ginibre inequality and the technical step are given in
Sections~\ref{sec:reverseinequality} and~\ref{sec:the_technical_step} respectively.

\subsubsection*{Reverse Ginibre inequality}

We now explain the main input for the upper bound
(Sections~\ref{sec:upper_bound_overview}--\ref{sec:the_technical_step}).
The following inequality is used in its proof.
Let $\Lambda_{m,n}$ denote the nearest-neighbour graph on $([-m,m]\times[-n,n])\cap\Z^2$.

\begin{theorem}[Reverse Ginibre inequality]
    \label{thm:reverse}
    There exists a universal constant $\crev>0$ with the following properties.
    Fix $\beta\in[0,\infty)$, fix $n\in\Z_{\geq 1}$, and fix $r\in\Z_{\geq 2}$ with $r\leq 2n$.
    Let $u=(0,n)$ and $v=(0,-n)$.
    Then
    \begin{equation}
        \langle \sigma_u\bar\sigma_v\rangle_{\Lambda_{r,n},\beta}
        \geq
        \crev^{n/r}
        \langle \sigma_u\bar\sigma_v\rangle_{\Lambda_{2n,n},\beta}.
    \end{equation}
\end{theorem}

Without the correction prefactor, the inequality would be true in the opposite direction by
the Ginibre inequality.
The correction prefactors are precisely what eventually contributes to the $c s^2/2$
term in the upper bound of Equation~\eqref{eq:main}.
This reverse Ginibre inequality can be proved by translating the two-point function into
a ratio of height-function partition functions via the Kadanoff--Ceva correspondence,
using the pushing lemma to modify the domain,
and then translating back to the XY model via the same correspondence (this is
done in Section~\ref{sec:reverseinequality}).
The percolation framework for the application of the pushing lemma to the height function
is described in Section~\ref{sec:the_cable_system}.

\section{Background}
\label{section:background}

\subsection{A phase transition driven by vortices}
Since the work of Peierls~\cite{Peierls_1936_IsingsModelFerromagnetism}, it is known that
lattice models (in dimension $d\geq 2$)
can undergo phase transitions.
Substantial effort has gone into classifying such transitions by analysing critical and
near-critical behaviour, as well as possible scaling limits.

This work considers the \emph{classical XY model} or \emph{classical rotor model} on $\Z^2$.
Mermin--Wagner~\cite{MerminWagner_1966_AbsenceFerromagnetismAntiferromagnetism} proved
that the 2D XY model does not undergo a magnetisation transition: the continuous spin
symmetry cannot be spontaneously broken, even at very low temperature.
McBryan--Spencer later gave a general power-law upper bound on the two-point
function~\cite{McBryanSpencer_1977_DecayCorrelationsSOnsymmetric}. This is very specific
to the two-dimensional setting: Fröhlich, Simon, and Spencer proved that the model \emph{does} exhibit
a magnetisation transition in dimension three and higher~\cite{FrohlichSimonSpencer_1976_InfraredBoundsPhase}.
This phenomenon is closely
related to the delocalisation of the Gaussian free field on
the square lattice graph $\Z^2$ (and, more generally, on recurrent
graphs), and localisation of the Gaussian free field
on $\Z^d$ for $d\geq 3$.
Around the time of the Mermin--Wagner theorem, numerical simulations suggested that the
qualitative behaviour
at $\beta\approx0$ does not match the behaviour at $\beta\approx\infty$,
suggesting that the model may undergo some sort of phase transition of a different
nature~\cite{StanleyKaplan_1966_PossibilityPhaseTransition}.
Berezinskii gave an early theoretical explanation for the occurrence of a phase
transition in this model.
His first work on this subject~\cite{Berezinskii_1971_DestructionLongrangeOrder}
considers the second-order expansion of the model
around the ground state. In this spin-wave, or Gaussian-free-field, approximation,
correlations can be exactly computed,
leading to polynomial decay of correlations at low temperatures (an indicator of critical
behaviour).
A second work~\cite{Berezinskii_1972_DestructionLongrangeOrder} considers the effect of
topological defects, or vortices, on the model, and predicts that they provide the driving
mechanism behind the phase transition.
Kosterlitz--Thouless~\cite{KosterlitzThouless_1973_OrderingMetastabilityPhase} later gave
a more detailed analysis of the effect of vortices: they gave an estimate of the critical
inverse temperature $\beta_c$ as well as a description of the near-critical window.
They also predicted that the phase transition is of infinite order
(that the relevant thermodynamic quantities are smooth but not analytic at $\beta_c$),
setting this phase transition somewhat apart from others (which are classified according
to the number of times they are differentiable).
This phase transition is now called the Berezinskii--Kosterlitz--Thouless (BKT) transition.

Fröhlich and Spencer gave a rigorous proof of a BKT phase transition:
they proved that the two-point function decays polynomially
at high $\beta$, in contrast with exponential decay at low
$\beta$~\cite{FrohlichSpencer_1981_KosterlitzThoulessTransitionTwodimensional}.
In their proof, they analyse in tandem the XY model
and a corresponding dual model, an integer-valued height function model
on the faces of the same square lattice graph.
Such duality transforms go back to the work of
Kramers--Wannier~\cite{KramersWannier_1941_StatisticsTwodimensionalFerromagnet_I,KramersWannier_1941_StatisticsTwodimensionalFerromagnet_II}
on the Ising model, and have been used in many different situations since.
One way to view the Fröhlich--Spencer argument is that it controls the vortex
contribution at high
$\beta$ (cf.~\cite{KharashPeled_2017_FrohlichSpencerProofBerezinskiiKosterlitzThouless}),
leading to polynomial decay of correlations
as in the spin-wave Ansatz of Berezinskii~\cite{Berezinskii_1971_DestructionLongrangeOrder}.

Garban and Sepúlveda~\cite{GarbanSepulveda_2023_QuantitativeBoundsVortex} give
quantitative bounds showing, in the Villain/Coulomb-gas setting, that vortex
fluctuations are at least comparable to spin-wave fluctuations in the regimes they
consider.
These results suggest that the second-order expansion of Berezinskii does not provide a
complete approximation of the model, even at
very low temperatures
(cf.~\cite{JoseKadanoffKirkpatrick_1977_RenormalizationVorticesSymmetrybreaking} for a
physical renormalisation-group argument pointing in this direction).
In~\cite{GarbanSepulveda_2023_StatisticalReconstructionGFF}, the authors approach the
related question from a different perspective, by isolating the contribution
of the spin wave to the spins in the model.
In that case, one can ask whether the original spin wave can be reconstructed from the spins
in the model.
This question is thought to be closely related to the driving mechanism behind the BKT transition,
and the authors show that reconstruction is possible at low temperatures, but not at high
temperatures.

\subsection{Relation to the height function}

A Kramers--Wannier-type duality transforms the XY model into a height function model,
where the height function is integer-valued and defined on the faces of the square
lattice
graph~\cite{FrohlichSpencer_1981_KosterlitzThoulessTransitionTwodimensional,Dubedat_2011_TopicsAbelianSpin,EngelenburgLis_2023_ElementaryProofPhase,AizenmanHarelPeled_2022_DepinningIntegerrestrictedGaussian,Lammers_2023_BijectingBKTTransition,EngelenburgLis_2025_DualityHeightFunctions}.
Over the past 25 years, advances in the understanding of height functions have also
influenced the study of the BKT transition.
A useful starting point for this line is the work of
Sheffield~\cite{Sheffield_2005_RandomSurfaces}, which initiated
the systematic application of percolation theory to the analysis of 2D height functions.
Subsequent work established delocalisation criteria and dichotomies for height
functions~\cite{Lammers_2022_HeightFunctionDelocalisation,LammersOtt_2024_DelocalisationAbsolutevalueFKGSolidonsolid,Lammers_2023_DichotomyTheoryHeight}.
Independently, Van
Engelenburg--Lis~\cite{EngelenburgLis_2023_ElementaryProofPhase} and
Aizenman--Harel--Peled--Shapiro~\cite{AizenmanHarelPeled_2022_DepinningIntegerrestrictedGaussian}
showed that the delocalisation transition of the height function implies the BKT phase
for the XY model. Together, these works provide a percolation-based proof
of the BKT transition, alternative to the original perturbative analysis of Fröhlich--Spencer.
Two further works clarify this connection.
It was established that the Kramers--Wannier duality exactly matches the BKT phase in the
XY model to the delocalised phase
of the height function~\cite{Lammers_2023_BijectingBKTTransition}.
Van Engelenburg and Lis~\cite{EngelenburgLis_2025_DualityHeightFunctions} studied this
duality further, by applying the duality transform to map
correlators in one model to disorder variables in the dual model,
in the spirit of Kadanoff--Ceva~\cite{KadanoffCeva_1971_DeterminationOperatorAlgebra}.
We use a boundary version of this mapping:
the correlators considered here lie \emph{on the boundary of a finite planar domain},
in which case the dual disorder variables are also located on the boundary of the domain,
and can be interpreted as a ratio between finite domain partition functions with
different boundary conditions.
The Kadanoff--Ceva mapping also uses
the Brydges--Fröhlich--Spencer loop and path representation of the
model~\cite{BrydgesFrohlichSpencer_1982_RandomWalkRepresentation}. The use of this
representation in the XY setting appears in recent work of Van Engelenburg and
Lis~\cite{EngelenburgLis_2023_ElementaryProofPhase}.
Further analysis of the loop representation may be useful for the XY model, as
random-current and loop methods have been useful in the context of the Ising model
(e.g.~\cite{AizenmanDuminil-CopinSidoravicius_2015_RandomCurrentsContinuity,AizenmanDuminil-Copin_2021_MarginalTrivialityScaling}).

In the context of the Villain model, an XY model with a modified potential,
the dual height function has a square potential, and is more tractable.
For this square-potential height function, Bauerschmidt, Park, and Rodriguez proved
convergence to the Gaussian free field in a high-temperature
regime~\cite{BauerschmidtParkRodriguez_2024_DiscreteGaussianModel,
BauerschmidtParkRodriguez_2024_DiscreteGaussianModela}
(the multiplicative constant is not identified explicitly).

The connection between surface tension, percolative descriptions of interfaces, and
roughening phenomena has a long history. In the context of Ising ferromagnets, Bricmont,
Fontaine and Lebowitz derived inequalities relating the surface tension of
phase-separation interfaces to percolative properties, and discussed consequences for the
roughening transition~\cite{BricmontFontaineLebowitz_1982_SurfaceTensionPercolation}.

\subsection{Universality of 2D height functions and inspiration from the six-vertex model}
A common expectation, supported by results for several height-function models, is that
2D integer-valued height functions with a local symmetric interaction satisfy a dichotomy:
either they are localised with a finite correlation length,
or they converge to $T\cdot\Gamma$, where $\Gamma$ is the Gaussian free field and $T$ is a
multiplicative constant
depending on the model, but always satisfying $T\geq c$ where $c>0$ is some universal
constant.
Such behaviour would be notable because height functions are dual to a broad family of
spin models of varying nature.

One case relevant here is the six-vertex model's height function.
Some results and relations to the random-cluster model were already mentioned in the discussion
following Theorem~\ref{thm:Main}.
The six-vertex model has an integrable structure, through the Bethe Ansatz,
which allows the explicit calculation of certain
quantities~\cite{Duminil-CopinSidoraviciusTassion_2017_ContinuityPhaseTransition,Duminil-CopinGagnebinHarel_2021_DiscontinuityPhaseTransition,Duminil-CopinKozlowskiKrachun_2022_SixvertexModelsFree,Duminil-CopinKajetanKozlowskiLammers_2026_GaussianFreeField}.
These formulas are compatible with the expected behaviour of their counterparts for the
XY model around the BKT transition, motivating the view of the six-vertex transition as
an integrable analogue of the BKT transition
(for example, the free energy and the random-cluster model's mass are smooth but not
analytic at the transition point).
The delocalisation transition can be analysed via various
routes~\cite{ChandgotiaPeledSheffield_2021_DelocalizationUniformGraph,Lis_2021_DelocalizationSixvertexModel,Lis_2022_SpinsPercolationHeight,Duminil-CopinHarelLaslier_2022_LogarithmicVarianceHeight,GlazmanPeled_2023_TransitionDisorderedAntiferroelectric,Duminil-CopinKarrilaManolescu_2024_DelocalizationHeightFunction,GlazmanLammers_2025_DelocalisationContinuity2D}.
We also point to~\cite{Dubedat_2011_TopicsAbelianSpin} for general background on
dualities in different spin systems.
Further research is necessary to understand how far this analogy extends, and whether it
can be used to predict further properties of the XY model.

\section{Kadanoff--Ceva correspondence}
\label{sec:kadanoff_ceva}

For the square box, the Kramers--Wannier expansion gives
$Z_{\Lambda_n,\beta}=Z_{\Lambda_n^*,\beta}^0$
with our normalisations
(\cite{KramersWannier_1941_StatisticsTwodimensionalFerromagnet_I,KramersWannier_1941_StatisticsTwodimensionalFerromagnet_II};
    see
also~\cite{FrohlichSpencer_1981_KosterlitzThoulessTransitionTwodimensional,Dubedat_2011_TopicsAbelianSpin,EngelenburgLis_2023_ElementaryProofPhase,Lammers_2023_BijectingBKTTransition,EngelenburgLis_2025_DualityHeightFunctions}).
Kadanoff and Ceva~\cite{KadanoffCeva_1971_DeterminationOperatorAlgebra} observed that
this partition-function identity extends to correlation functions.
Inserting a spin correlation on one side of the duality produces \emph{disorder
operators} on the other side.
Here we only need the boundary version of this statement: XY correlations between
vertices on the outer face correspond to changing the boundary condition of the dual
height function.
We first fix the planar-duality convention used below.

\begin{definition}[Planar duality; Figure~\ref{fig:dual_graph}, \textsc{Right}]
    \label{def:planar_duality}
    Let $G$ be a leafless connected finite planar graph embedded in the plane; we call it
    the \emph{primal graph}.
    Write $\partial G\subset V(G)$ for the set of vertices incident to the outer face.
    The corresponding \emph{dual graph} $G^*$ is defined as follows.
    \begin{itemize}
        \item Its vertex set $V(G^*)$ is the disjoint union of:
            \begin{itemize}
                \item The set of \emph{inner} faces of $G$,
                \item The set of edges of $G$ incident to the outer face.
            \end{itemize}
        \item The edge set $E(G^*)$ is in natural bijection with $E(G)$.
        \item The graph boundary $\partial G^*\subset V(G^*)$ is defined as the set of
            dual vertices corresponding to the edges incident to the outer face of $G$.
    \end{itemize}

    This gives an efficient set-theoretic definition of the dual graph.
    We typically use it through its natural planar embedding,
    see Figure~\ref{fig:dual_graph}.
\end{definition}

\begin{lemma}[Kadanoff--Ceva correspondence]
    \label{lem:kadanoff_ceva}
    Let $G$ be any leafless connected finite planar graph endowed with edge weights
    $(J_{uv})_{uv\in E(G)}\subset[0,\infty)$.
    Let $G^*$ denote the dual graph, and extend $J$ to $E(G^*)$ via the natural duality bijection.
    Then for any boundary height function $\zeta:\partial G^*\to\Z$,
    we have
    \begin{equation}
        \label{eq:kadanoff_ceva}
        Z_{G^*,J}^\zeta =
        Z_{G,J}
        \Big\langle \prod_{u\in\partial G} (\sigma_u)^{\zeta_{u_+}-\zeta_{u_-}}\Big\rangle_{G,J},
    \end{equation}
    where $u_-$ and $u_+$ are the first dual vertices adjacent to $u$ in the clockwise
    and counterclockwise directions, respectively (Figure~\ref{fig:dual_graph}).
    Thus a boundary height jump across $u$ becomes the spin insertion at $u$.
    Adding a constant to $\zeta$ changes neither side of the identity.

    In particular, for $s\in(0,1)$, we have
    \begin{equation}
        \label{eq:kadanoff_ceva_slope}
        \frac{Z_{\Lambda_n^*,\beta}^s}{Z_{\Lambda_n^*,\beta}^{0}}
        =
        \Big\langle
        \Big[
            \prod\nolimits_{u\in T_n(s)} \bar\sigma_u
        \Big]\Big[
            \prod\nolimits_{v\in B_n(s)} \sigma_v
        \Big]
        \Big\rangle_{\Lambda_n,\beta},
    \end{equation}
    where the top and bottom source sets are
    \begin{equation}
        \label{eq:source_sets}
        T_n(s)=(I_s\cap[-n,n])\times\{n\},
        \qquad
        B_n(s)=(I_s\cap[-n,n])\times\{-n\},
    \end{equation}
    and
    \begin{equation}
        \label{eq:jump_locations}
        I_s:=\{k\in\Z:\lfloor s(k+\tfrac12)\rfloor>\lfloor s(k-\tfrac12)\rfloor\}.
    \end{equation}
\end{lemma}

\section{Lower bound on the free energy}
\label{sec:lower_bound}

We first prove the following auxiliary inequality.

\begin{lemma}
    Fix $n$ and let $\zeta:\partial\Lambda_n^*\to\Z$ denote any boundary condition.
    Write $\zeta^+:=\zeta\vee 0$ and $\zeta^-:=(-\zeta)^+$ for the positive and negative
    parts of $\zeta$ respectively,
    so that $\zeta=\zeta^+-\zeta^-$.
    Then for any $\beta\in[0,\infty)$, we have
    \begin{equation}
        \label{eq:boundary_fkg}
        Z_{\Lambda_n^*,\beta}^{\zeta^+} \cdot Z_{\Lambda_n^*,\beta}^{-\zeta^-}
        \geq
        Z_{\Lambda_n^*,\beta}^\zeta \cdot Z_{\Lambda_n^*,\beta}^0.
    \end{equation}
\end{lemma}

\begin{proof}
    This is the FKG inequality for height functions with convex potentials,
    once the ambient probability measure is set up correctly.
    The case $\beta=0$ is trivial, so assume $\beta>0$.
    If $\zeta\geq 0$ or $\zeta\leq 0$ then the lemma is trivial; we study the case that
    $\zeta$ takes both positive and negative values below.
    Write $\partial\Lambda_n^*$ as the disjoint union of
    $Q_-:=\{\zeta<0\}$, $Q_+:=\{\zeta>0\}$, and $Q_0:=\{\zeta=0\}$.
    Write $\Omega'$ for the set of configurations
    \begin{equation}
        \Omega' := \left\{ h\in \Z^{\Lambda_n^*} :
            h|_{Q_-} \in \{\zeta|_{Q_-},0\},\,
            h|_{Q_+} \in \{0,\zeta|_{Q_+}\},\,
        h|_{Q_0}=0 \right\}.
    \end{equation}
    Notice that $\Omega'$ has the structure of a distributive lattice.
    We equip $\Omega'$ with the probability measure $\mu$ defined by
    \begin{equation}
        \mu[\{h\}] \propto \prod_{xy\in E(\Lambda_n^*)} w_\beta(h(y)-h(x)).
    \end{equation}
    Since $w_\beta$ is log-concave, $\mu$ satisfies the FKG inequality on $\Omega'$
    with respect to the natural partial order.
    Let $A:=\{h|_{Q_-}=0\}$ and $B:=\{h|_{Q_+}=\zeta|_{Q_+}\}$.
    Both events are increasing, hence
    \begin{equation}
        \mu[A\cap B]\mu[A^c\cap B^c]\geq
        \mu[A^c\cap B]\mu[A\cap B^c].
    \end{equation}
    Up to normalisation, those four probabilities correspond precisely to the four
    partition functions in Equation~\eqref{eq:boundary_fkg}.
\end{proof}

Since $Z_{\Lambda_n^*,\beta}^\zeta = Z_{\Lambda_n^*,\beta}^{\zeta+a}$ for any
integer constant $a$,
we may iterate the above lemma over the level sets of the boundary condition and then
apply the Kadanoff--Ceva identity (Equation~\eqref{eq:kadanoff_ceva_slope}) to obtain
\begin{equation}
    \label{eq:lower_bound_kadanoff_ceva_new_FKG_boundary}
    \frac{Z_{\Lambda_n^*,\beta}^s}{Z_{\Lambda_n^*,\beta}^0}
    \leq
    \prod_{u\in T_n(s)} \langle \bar\sigma_u \sigma_{u-2ne_2}\rangle_{\Lambda_n,\beta}.
\end{equation}
By the Ginibre inequality, we may further upper bound this quantity by replacing the
finite graph $\Lambda_n$ by the graph limit, yielding
\begin{equation}
    \label{eq:lower_bound_kadanoff_ceva_new_FKG_boundary_2}
    \frac{Z_{\Lambda_n^*,\beta}^s}{Z_{\Lambda_n^*,\beta}^0}
    \leq
    (\langle \sigma_0 \bar\sigma_{2ne_1}\rangle_{\Z^2,\beta})^{|T_n(s)|}.
\end{equation}
For fixed $\beta$ and $s$, the right-hand side of this equation has
asymptotics
\begin{equation}
    \label{eq:lower_bound_asymptotic}
    \big(e^{-2n\cdot\frakm_\XY(\beta)(1+o(1))}\big)^{2ns+O(1)}
    =
    e^{-4sn^2\cdot\frakm_\XY(\beta)(1+o(1))}
\end{equation}
as $n\to\infty$ by the definition of the mass
and the fact that $|T_n(s)|=2ns+O(1)$.
Combining this asymptotic estimate with
the definition of the free energy (Equation~\eqref{eq:free_energy}), we obtain the
desired lower bound $f_\beta(s)-f_\beta(0)\geq s\cdot\frakm_\XY(\beta)$ in Theorem~\ref{thm:Main}.

\section{Upper bound on the free energy}
\label{sec:upper_bound_overview}

The reverse Ginibre inequality
(Theorem~\ref{thm:reverse})
plays an essential role in the proof of the upper bound in Theorem~\ref{thm:Main}.
It is used in the following way.
The Ginibre inequality tells us that the two-point function is increasing
in the domain, but suppose for a moment that the Ginibre inequality
holds true for multi-point correlation functions as well.
This ``generalised Ginibre inequality'' would imply that
(when $\Lambda'_{n}(s)\subset\Lambda_{n,n}$)
\begin{multline}
    \label{eq:multi_point_ginibre}
    \frac{Z_{\Lambda_n^*,\beta}^s}{Z_{\Lambda_n^*,\beta}^0}
    =
    \Big\langle\Big[\prod\nolimits_{u\in
    T_n(s)}\bar\sigma_u\Big]\Big[\prod\nolimits_{v\in
    B_n(s)}\sigma_v\Big]\Big\rangle_{\Lambda_n,\beta}
    \\
    \geq
    f(n,s)\cdot
    \Big\langle\Big[\prod\nolimits_{u\in T_n(s)}\bar\sigma_u\Big]\Big[\prod\nolimits_{v\in
    B_n(s)}\sigma_v\Big]\Big\rangle_{\Lambda'_{n}(s),\beta},
\end{multline}
where $f(n,s)=1$ and
where $\Lambda'_{n}(s)$ is the domain
\begin{equation}
    \bigcup_{k\in I_s\cap[-n,n]}\left(\Lambda_{1/(20s),n}+(k,0)\right)
    \subset\Lambda_{n,n}.
\end{equation}
This domain is a union of long strips of width $1/(10s)$ around the vertical lines
containing the jump locations, and is contained in $\Lambda_{n,n}$
(we pick $n$ appropriately so the jumps are not too close to the vertical boundary,
and it is indeed true that $\Lambda'_{n}(s)\subset\Lambda_{n,n}$).
If the strip width is not an integer, we simply round it down to the nearest integer,
which does not change the asymptotics.
Suppose that Equation~\eqref{eq:multi_point_ginibre} is true (with $f=1$).
The correlation function on the right in Equation~\eqref{eq:multi_point_ginibre}
is just a product of two-point functions (since every connected component
of the domain $\Lambda'_{n}(s)$ contains exactly one source and one sink).
Thus we may use the reverse Ginibre inequality to obtain
\begin{multline}
    \label{eq:multi_point_ginibre_with_reverse}
    \Big\langle\Big[\prod\nolimits_{u\in T_n(s)}\bar\sigma_u\Big]\Big[\prod\nolimits_{v\in
    B_n(s)}\sigma_v\Big]\Big\rangle_{\Lambda'_{n}(s),\beta}
    \\
    =
    \left(\langle
    \bar{\sigma}_{(0,n)}\sigma_{(0,-n)}\rangle_{\Lambda_{1/(20s),n},\beta}\right)^{|I_s\cap[-n,n]|}
    \geq \left(\crev^{20sn} \cdot \langle
    \bar{\sigma}_{(0,n)}\sigma_{(0,-n)}\rangle_{\Lambda_{2n,n},\beta}\right)^{|I_s\cap[-n,n]|},
\end{multline}
where the inequality is the reverse Ginibre inequality (Theorem~\ref{thm:reverse}).
Since $|I_s\cap[-n,n]|=2ns+O(1)$ and since the two-point function has asymptotics
$e^{-2n\cdot\frakm_\XY(\beta)(1+o(1))}$~\cite[Claim 8]{Lammers_2023_BijectingBKTTransition},
we get
\begin{equation}
    \frac{Z_{\Lambda_n^*,\beta}^s}{Z_{\Lambda_n^*,\beta}^0}
    \geq
    f(n,s)\cdot
    \crev^{40s^2n^2(1+o(1))}\cdot
    e^{-4sn^2\cdot\frakm_\XY(\beta)(1+o(1))}
\end{equation}
as $n\to\infty$.
It turns out that we cannot prove Equation~\eqref{eq:multi_point_ginibre} with
$f(n,s)=1$, but we can prove it with $f(n,s)=c^{n^2s^2}$ for some universal constant $c\in(0,1)$.
The previous display then yields the desired upper bound in Theorem~\ref{thm:Main}
(Equation~\eqref{eq:main}).

The remaining sections rigorously justify the two gaps in the proof (the reverse Ginibre
    inequality of Theorem~\ref{thm:reverse}, as well as
Equation~\eqref{eq:multi_point_ginibre} with the nontrivial correction factor).
Section~\ref{sec:the_cable_system} contains a description of a framework
that enables the application of percolation theory arguments to our height functions.
Section~\ref{sec:reverseinequality} uses this framework to prove the reverse Ginibre inequality.
Finally, Section~\ref{sec:the_technical_step} establishes
Equation~\eqref{eq:multi_point_ginibre} with the correction factor $f(n,s)=c^{n^2s^2}$.

\section{The cable system extension of the height function}
\label{sec:the_cable_system}

Several models of statistical mechanics admit a useful extension to the cable graph.
This idea goes back at least
to~\cite[Page~494]{Berezinskii_1971_DestructionLongrangeOrder} in the context of the Villain model,
and is used extensively for the discrete Gaussian free field, starting with
Lupu~\cite{Lupu_2016_LoopClustersRandom},
and other models with a quadratic interaction
(cf.~\cite{DubedatFalconet_2022_RandomClustersVillain}).
Similar ideas apply to the height function of the XY model.

Let $\BC:=((\Z+\frac12)\times\R)\cup(\R\times(\Z+\frac12))$ be the cable system
associated to the dual graph $(\Z+\frac12)^2$ of the square lattice $\Z^2$.
It is endowed with the subspace topology and the Lebesgue measure,
denoted $\diff x$ or $\operatorname{Leb}(\blank)$.
A \emph{continuum domain}, or simply a \emph{domain}, is a bounded open subset
$\Gamma\subset\BC$ consisting of finitely many connected components.
For fixed $n,m\in\Z_{\geq 1}$, let
$\Gamma_{n,m}:=(-n-\frac12,n+\frac12)\times(-m-\frac12,m+\frac12)\cap\BC$ denote the
continuum domain associated with the duality pair $(\Lambda_{n,m},\Lambda_{n,m}^*)$.
We write $\Gamma_n:=\Gamma_{n,n}$.

The height function can be extended from $\Lambda_n^*$ to
$\bar\Gamma_n\supset\Lambda_n^*$ in a natural way.
The construction is essentially as follows:
first sample two independent sets of \emph{up} and \emph{down} jumps from Poisson point
processes on $\Gamma$ with intensity $\frac\beta2\operatorname{Leb}|_\Gamma$,
then condition on the event that these jumps are the gradient of a height function with
the desired boundary condition.
For the following definition, if $x\in\Gamma\setminus(\Z+\frac12)^2$, write
\begin{equation}
    h(x^-):=
    \begin{cases}
        \lim_{\lambda\to 0^+} h(x-\lambda e_1) &\text{if $x$ lies on a horizontal line segment}, \\
        \lim_{\lambda\to 0^+} h(x-\lambda e_2) &\text{if $x$ lies on a vertical line segment}.
    \end{cases}
\end{equation}
Define $h(x^+)$ similarly.

More precisely:
\begin{itemize}
    \item Consider a continuum domain $\Gamma$,
    \item Let $\BP_{\Gamma,\beta}$ denote the Poisson point process on $\Gamma\times\{+,-\}$
        with intensity $\frac\beta2\operatorname{Leb}\times(\delta_++\delta_-)$, and
        denote the two random sets of marks by $(\Pi_+,\Pi_-)\in\calP(\Gamma)\times\calP(\Gamma)$,
    \item Let $h:\bar\Gamma\to\frac12\Z$ denote any function such that:
        \begin{itemize}
            \item $h$ is integer-valued and continuous on $\bar\Gamma\setminus(\Pi_+\cup\Pi_-)$,
            \item For every $x\in\Pi_+$, we have $h(x^-)+\frac12 = h(x) = h(x^+)-\frac12$,
            \item For every $x\in\Pi_-$, we have $h(x^-)-\frac12 = h(x) = h(x^+)+\frac12$,
        \end{itemize}
    \item Let $W$ denote the event that there exists such a height function $h$.
\end{itemize}

Thus, on $W$, the \emph{gradient} of $h$ is uniquely determined on each connected
component of $\bar\Gamma$.
This gradient can be turned into a height function by either choosing a root vertex in
each connected component of $\bar\Gamma$,
or by imposing boundary conditions on $\partial\Gamma$ and restricting to the event that
the gradient
matches the gradient of that boundary condition.

We shall also write $\BM_{\Gamma,\beta}$ for the non-normalised version of
$\BP_{\Gamma,\beta}$ restricted to $W$.
More precisely, $\BM_{\Gamma,\beta}:=\ind{W} e^{\beta\operatorname{Leb}(\Gamma)}\BP_{\Gamma,\beta}$.
Finally, we write $\BM_{\Gamma,\beta}^\zeta$ for the restriction of $\BM_{\Gamma,\beta}$
to the event $\{h|_{\partial\Gamma}=\zeta\}$.

Write $Z_{\Gamma,\beta}$ for the total mass of $\BM_{\Gamma,\beta}$ and
$Z_{\Gamma,\beta}^\zeta$ for the mass of $\BM_{\Gamma,\beta}^\zeta$.
Even though those two measures are not probability measures,
we shall still occasionally speak of ``conditioned'' measures.
The conditioned measures are normalised and thus probability measures.

\begin{proposition}
    For any $\beta\in(0,\infty)$, $n\in\Z_{\geq 1}$, and $\zeta\in\Z^{\partial\Lambda_n^*}$,
    the push-forward of $\BM_{\Gamma_n,\beta}^\zeta$ under $h\mapsto h|_{\Lambda_n^*}$ is
    \begin{equation}
        Z_{\Lambda_n^*,\beta}^\zeta
        \cdot
        \mu_{\Lambda_n^*,\beta}^\zeta.
    \end{equation}
    This means that the measure $\BM_{\Gamma_n,\beta}^\zeta$ encodes
    simultaneously the law of $h|_{\Lambda_n^*}$ as well as
    the partition function.
\end{proposition}

\begin{proof}
    On each cable edge, the numbers of up and down jumps are independent Poisson random
    variables with mean $\beta/2$.
    After multiplying by the factor $e^{\beta\operatorname{Leb}(\Gamma_n)}$, the mass of
    a configuration with $i$ up jumps and $j$ down jumps on that edge is
    \begin{equation}
        \frac{(\beta/2)^i}{i!}\frac{(\beta/2)^j}{j!}.
    \end{equation}
    Summing over all pairs $(i,j)$ with $i-j=h_x-h_y$ gives exactly the edge weight
    $w_\beta(h_x-h_y)$ from Definition~\ref{def:dual_height_function}.
    Taking the product over all dual edges gives the claimed marginal and the correct total mass.
\end{proof}

The cable system provides us with a few convenient tools that come almost for free.
They are summarised in the following proposition.

\begin{proposition}
    Consider the measure $\BM_{\Gamma,\beta}$.
    Then the following properties hold.
    \begin{itemize}
        \item \textbf{Intermediate value theorem.}
            For $\BM_{\Gamma,\beta}$-almost every configuration,
            the following property holds true.
            Let $h$ denote a height function (defined either by choosing a root vertex or
            by imposing boundary conditions).
            If $\gamma\subset\bar\Gamma$ is any continuous path connecting two points
            $x,y\in\bar\Gamma$,
            and if $k\in\Z$ is any integer between $h(x)$ and $h(y)$,
            then there exists a point $z\in\gamma$ such that $h(z)=k$.
        \item \textbf{Markov property.}
            Let $h$ denote a height function (defined either by choosing a root vertex or
            by imposing boundary conditions).
            Then for any closed set $A\subset\bar\Gamma$ with $\partial\Gamma\subset A$
            and with $\Gamma\setminus A$ a finite union of continuum domains, the
            conditional law of $h$ on $\Gamma\setminus A$, given $h|_A$, is
            \begin{equation}
                \bigotimes_{\Gamma'\in\calC(\Gamma\setminus
                A)}\frac1{Z_{\Gamma',\beta}^{h|_{\partial
                \Gamma'}}}\BM_{\Gamma',\beta}^{h|_{\partial \Gamma'}}.
            \end{equation}
    \end{itemize}
\end{proposition}

\begin{proof}
    The intermediate value theorem follows from the definition of $h$,
    which jumps by one and records the midpoint value at each jump.
    The Markov property follows from the fact that the \emph{a priori} randomness
    comes from a Poisson point process, whose restrictions to disjoint Borel sets are independent.
\end{proof}

The independence of the underlying Poisson point process
can also be used to make sense of the conditional law following an exploration process.
This is described now.
First, define
\begin{equation}
    \calE_a := \{x\in\bar\Gamma:h(x)\neq a\};
    \qquad
    \calE_a(S) := S\cup \Big(\cup_{A\in\calC(\calE_a),\, A\cap S\neq\emptyset}A\Big).
\end{equation}

\begin{proposition}[Exploration process]
    \label{prop:ExplorationProcess}
    Consider the measure $\BM_{\Gamma,\beta}^\zeta$
    with boundary condition $\zeta$.
    Fix $a\in\Z$ and $S\subset\BC$.
    Set $\Gamma':=\Gamma\setminus\bar\calE_a(S)$.
    Then conditional on $\calE_a(S)$, the law of $h$ in $\bar\Gamma'$ is
    $(Z_{\Gamma',\beta}^{\zeta'})^{-1}\BM_{\Gamma',\beta}^{\zeta'}$,
    where $\zeta'\in\Z^{\partial\Gamma'}$ is the boundary condition induced by the exploration:
    it agrees with the original boundary condition on $\partial\Gamma'\cap\partial\Gamma$
    and is equal to $a$ on the newly created boundary $\partial\Gamma'\setminus\partial\Gamma$.
\end{proposition}

\begin{proof}
    For a realised exploration, write $A=\bar\calE_a(S)$.
    The explored set is determined by the Poisson marks in $A$ together with the
    information that the height is equal to $a$ on the inner boundary separating $A$ from
    $\Gamma':=\Gamma\setminus A$.
    Since the underlying Poisson point process has independent restrictions to the
    disjoint sets $A$ and $\Gamma'$, conditioning on this explored information does not
    change the law of the Poisson marks inside $\Gamma'$ except through the boundary
    values just described.
    On each connected component of $\Gamma'$, the remaining marks are therefore
    distributed as the original Poisson process conditioned to form a valid height
    function with boundary condition $\zeta'$.
    This is precisely the normalised measure
    $(Z_{\Gamma',\beta}^{\zeta'})^{-1}\BM_{\Gamma',\beta}^{\zeta'}$.
\end{proof}

We now want to extend the Kadanoff--Ceva construction to the more general continuum domains.
We call a continuum domain $\Gamma$ \emph{simply connected} if both $\Gamma$ and
$\BC\setminus\Gamma$ are connected,
and if $\Gamma$ contains at least one vertex of $(\Z+\frac12)^2$.
The dual graph $\Gamma^*$ of a simply connected domain $\Gamma$ is particularly simple:
it is the union of all primal squares in $\Z^2$ around all dual vertices in the set
$\Gamma\cap(\Z+\frac12)^2$ (Figure~\ref{fig:dual_graph}, \textsc{Right}).
This graph is connected and leafless, and it is a weighted graph: the weight $J_{xy}$ of
an edge $xy$ is $\operatorname{Leb}(\Gamma\cap L_{xy})$,
where $L_{xy}\subset\BC$ is the unit line segment perpendicular to the primal edge $xy$.

\begin{lemma}[Kadanoff--Ceva correspondence for continuum domains]
    \label{lem:kadanoff_ceva_specific}
    Let $\Gamma$ denote any simply connected continuum domain, and fix $\beta\in[0,\infty)$.
    Let $\Gamma^*\subset\Z^2$ denote the dual graph, and define $J$ as specified above.
    Then for any boundary height function $\zeta:\partial \Gamma\to\Z$,
    we have
    \begin{equation}
        \label{eq:kadanoff_ceva_specific}
        Z_{\Gamma,\beta}^\zeta =
        Z_{\Gamma^*, J\beta}
        \Big\langle \prod_{u\in\partial \Gamma^*}
        (\sigma_u)^{\zeta_{u_+}-\zeta_{u_-}}\Big\rangle_{\Gamma^*,J\beta},
    \end{equation}
    where $u_-$ and $u_+$ are the first points in $\partial\Gamma$ adjacent to $u$ in the
    clockwise and counterclockwise direction, respectively
    (cf. Figure~\ref{fig:dual_graph}, \textsc{Right}).
\end{lemma}

\begin{proof}
    First integrate out the positions of the Poisson marks on each set $\Gamma\cap L_{xy}$.
    If $i$ up jumps and $j$ down jumps occur there, their contribution to the
    non-normalised cable measure is
    \begin{equation}
        \frac{(\beta J_{xy}/2)^i}{i!}\frac{(\beta J_{xy}/2)^j}{j!},
    \end{equation}
    and the net height difference across the corresponding dual edge is $i-j$.
    Summing over all such pairs with fixed difference gives the discrete height weight
    $w_{\beta J_{xy}}$.
    Hence $Z_{\Gamma,\beta}^\zeta$ is the partition function of the discrete height
    function on the height graph associated with $\Gamma$, with coupling constants $\beta
    J$ and boundary condition $\zeta$.
    Applying Lemma~\ref{lem:kadanoff_ceva} to the weighted planar graph $\Gamma^*$ gives
    the stated identity, with the displayed sign convention for the clockwise and
    counterclockwise boundary neighbours.
\end{proof}

\section{Proof of the reverse Ginibre inequality (Theorem~\ref{thm:reverse})}
\label{sec:reverseinequality}

\subsection{Definition of a two-point function in arbitrary domains}

Now consider $n\in\Z_{\geq 1}$ large, and consider $\beta$ fixed.
We shall consider simply connected continuum domains $\Gamma$ with the property that
\begin{equation}\label{eq:cable_property}
    \BC\cap([-1,1]\times(-n-\tfrac12,n+\tfrac12))\subset
    \Gamma\subset\BC\cap(\R\times(-n-\tfrac12,n+\tfrac12)).
\end{equation}
For any such domain $\Gamma$, we shall write ``$00$'' for the boundary height function
that is identically zero on $\partial\Gamma$,
and ``$01$'' for the boundary height function that is zero on the left of the vertical
axis $\{0\}\times\R$ and one on the right of that axis.
Let $u:=(0,n)$ and $v:=(0,-n)$.
Then by the Kadanoff--Ceva correspondence (Lemma~\ref{lem:kadanoff_ceva_specific}), we have
\begin{equation}
    \label{eq:two_point_cable_system}
    \frac{Z_{\Gamma,\beta}^{01}}{Z_{\Gamma,\beta}^{00}}=
    \frac{
        \BM_{\Gamma,\beta}[\{h|_{\partial\Gamma}=01\}]
    }{
        \BM_{\Gamma,\beta}[\{h|_{\partial\Gamma}=00\}]
    }
    =
    \langle \sigma_u\bar\sigma_v\rangle_{\Gamma^*,J(\Gamma)\beta}
    =:\Xi(\Gamma).
\end{equation}
This identity holds for any such simply connected continuum domain $\Gamma$.
By the Ginibre inequality, $\Xi$ is an increasing function of $\Gamma$.
The idea of this section is to use results from percolation theory (and in particular the
Russo--Seymour--Welsh theory) to
bound the sensitivity of $\Xi$ to changes in $\Gamma$.

\subsection{Random boundary conditions}

Fix an appropriate domain $\Gamma$.
Consider the probability measure
\begin{equation}
    \nu_{\Gamma}:=\BM_{\Gamma,\beta}[\blank|\{h|_{\partial\Gamma}\in\{00,01\}\}].
\end{equation}
This means that $\nu_{\Gamma}$ is obtained by sampling $\Pi_+$ and $\Pi_-$ from a Poisson
point process on $\Gamma$,
and then conditioning on the event that $h$ is a valid height function with boundary
condition \emph{either} $00$ or $01$.
The boundary condition is thus random, and we write $a$ for the $\{0,1\}$-valued random
variable encoding the height of the boundary on the right of the $y$-axis.
This probability measure encodes the information necessary to determine the two-point
function of Equation~\eqref{eq:two_point_cable_system}, since
\begin{equation}
    \label{eq:Xi_new}
    \Xi(\Gamma) = \frac{\nu_\Gamma[\{a=1\}]}{\nu_\Gamma[\{a=0\}]}.
\end{equation}

Recall that $\Xi(\blank)$ is only defined for (particular) simply connected domains.
For any continuum domain $\Omega$ satisfying \eqref{eq:cable_property}, write $D(\Omega)$
for the connected component of $\Omega$ carrying the two boundary jumps at $u$ and $v$,
whenever this component is simply connected.
We shall only apply this notation in that case.

\subsection{Sensitivity bound in terms of percolation}

For any $r\in\R$, introduce the half-planes $\H^-_r:=(-\infty,r)\times\R$ and
$\H^+_r:=(r,\infty)\times\R$.
For $\zeta\in\{00,01\}$, write
\begin{equation}
    \tilde\mu_{\Gamma,\beta}^{\zeta}:=
    \frac{1}{Z_{\Gamma,\beta}^{\zeta}}\BM_{\Gamma,\beta}^{\zeta}
\end{equation}
for the corresponding normalised measure.

\begin{lemma}
    \label{lem:exploration_bound_ratio_bound}
    Fix $n$, $\beta$, $\Gamma$, and $\nu_\Gamma$ as in the previous two subsections.
    Fix some arbitrary $r<-1$.
    Then
    \begin{equation}
        \label{eq:percolationbound}
        \frac{\Xi(D(\Gamma\cap\H^+_{r}))}{\Xi(\Gamma)} \geq
        \tilde\mu_{\Gamma,\beta}^{01}[\{\calE_0(\H^-_{r})\subset \H^-_{-1}\}]
        ,
    \end{equation}
    where $\calE_0(\blank)$ is the exploration set defined in
    Proposition~\ref{prop:ExplorationProcess}.
\end{lemma}

\begin{proof}
    Under $\nu_\Gamma$, conditioning on $\{a=1\}$ gives the normalised measure
    $\tilde\mu_{\Gamma,\beta}^{01}$.
    By Equation~\eqref{eq:Xi_new}, it is enough to prove the stronger estimate
    \begin{equation}
        \label{eq_exploration_bound_ratio_bound_1}
        \Xi(D(\Gamma\cap\H^+_{r})) \geq
        \frac{\nu_\Gamma[\{\calE_0(\H^-_{r})\subset
        \H^-_{-1}\}\cap\{a=1\}]}{\nu_\Gamma[\{\calE_0(\H^-_{r})\subset \H^-_{-1}\}\cap\{a=0\}]}.
    \end{equation}
    We claim that it suffices to prove that, for any
    $\H^-_{r}\cap\Gamma\subset\calA\subset \H^-_{-1}$,
    we have
    \begin{equation}
        \label{eq_exploration_bound_ratio_bound_2}
        \Xi(D(\Gamma\cap\H^+_{r})) \geq
        \frac{\nu_\Gamma\big[\{a=1\}\big|\{\calE_0(\H^-_{r})=\calA\}\big]}{\nu_\Gamma\big[\{a=0\}\big|\{\calE_0(\H^-_{r})=\calA\}\big]}.
    \end{equation}
    Indeed, Equation~\eqref{eq_exploration_bound_ratio_bound_1} then follows by integrating
    the numerator and denominator of the right-hand side of
    Equation~\eqref{eq_exploration_bound_ratio_bound_2} over $\calA$.
    We therefore focus on proving Equation~\eqref{eq_exploration_bound_ratio_bound_2}.

    Let us focus on the conditional measure $\nu_\Gamma[\blank|\{\calE_0(\H^-_{r})=\calA\}]$.
    This measure comes \emph{a priori} from a measure of independent Poisson jumps,
    conditioned on:
    \begin{itemize}
        \item The jumps forming a consistent height function with boundary condition $00$ or $01$,
        \item The height function equals zero on $\partial(\calA\cap\Gamma)$,
        \item Some event measurable in terms of what happens on the interior of $\calA$.
    \end{itemize}
    By independence of the Poisson point process, conditioning on the zeros on the
    boundary of $\calA\cap\Gamma$ makes
    its behaviour on $\calA$ independent from its behaviour on $\Gamma\setminus\calA$,
    and in particular the distribution of $h|_{\Gamma\setminus\calA}$ is given by
    $\nu_{\Gamma\setminus\calA}$.
    All components other than $D(\Gamma\setminus\calA)$ have constant boundary condition
    whether $a=0$ or $a=1$, and
    therefore cancel from the ratio below.
    The ratio on the right-hand side in
    Equation~\eqref{eq_exploration_bound_ratio_bound_2} therefore equals
    \begin{equation}
        \Xi(D(\Gamma\setminus\calA)).
    \end{equation}
    Since $D(\Gamma\setminus\calA)\subset D(\Gamma\cap\H^+_r)$ and $\Xi$ is increasing in
    the domain,
    this is indeed smaller than $\Xi(D(\Gamma\cap\H^+_r))$.
\end{proof}

\subsection{Application of the pushing lemma (Russo--Seymour--Welsh theory)}

\begin{proof}[Proof of Theorem~\ref{thm:reverse}]
    Let $r\in\Z_{\geq2}$ with $r\leq 2n$, and set
    \begin{equation}
        \Gamma:=\BC\cap((-2n-\tfrac12,2n+\tfrac12)\times(-n-\tfrac12,n+\tfrac12)),
        \qquad
        r_-:=-r-\tfrac12.
    \end{equation}
    The right-hand side of Equation~\eqref{eq:percolationbound}
    encodes the probability that the zeros vertically
    cross the rectangle $[r_-,-1]\times[-n,n]$.
    Exactly this probability was studied in~\cite[Pushing lemma
    (Lemma~7.3)]{Lammers_2023_DichotomyTheoryHeight},
    where it is proved that
    \begin{equation}
        \tilde\mu_{\Gamma,\beta}^{01}[\{\calE_0(\H^-_{r_-})\subset \H^-_{-1}\}]
        \geq c^{n/r}
    \end{equation}
    for some universal constant $c>0$. As $D(\Gamma\cap\H^+_{r_-})=\Gamma\cap\H^+_{r_-}$,
    Lemma~\ref{lem:exploration_bound_ratio_bound} implies that
    \begin{equation}
        \frac
        {\langle \bar{\sigma}_u\sigma_v\rangle_{(\Gamma\cap\H^+_{r_-})^*,J(\Gamma\cap
        \H^+_{r_-})\beta}}
        {\langle \bar\sigma_u\sigma_v\rangle_{\Gamma^*,J(\Gamma)\beta}}
        \geq c^{n/r}.
    \end{equation}
    By symmetry we may apply the same inequality on the intersection with $\H^-_{r}$, yielding
    \begin{equation}
        \frac
        {\langle \bar\sigma_u\sigma_v\rangle_{(\Gamma')^*,J(\Gamma')\beta}}
        {\langle
        \bar\sigma_u\sigma_v\rangle_{(\Gamma\cap\H^+_{r_-})^*,J(\Gamma\cap\H^+_{r_-})\beta}}
        \geq c^{n/r},
    \end{equation}
    where $\Gamma':=\Gamma\cap((-r-\tfrac12,r+\tfrac12)\times\R)$.
    Multiplying the two inequalities gives
    \begin{equation}
        \frac
        {\langle \bar\sigma_u\sigma_v\rangle_{(\Gamma')^*,J(\Gamma')\beta}}
        {\langle \bar\sigma_u\sigma_v\rangle_{\Gamma^*,J(\Gamma)\beta}}
        \geq c^{2n/r}.
    \end{equation}
    Since $\Gamma^*=\Lambda_{2n,n}$ and $(\Gamma')^*=\Lambda_{r,n}$, this is the desired
    inequality with $\crev:=c^2$.
\end{proof}

\begin{remark}
    When $r$ is proportional to $n$, the term $(c)^{2n/|r|}$ is
    constant and the two-point functions for $\Lambda_{2n,n}$ and $\Lambda_{r,n}$ are of
    the same order. However, when $r$ is constant as $n\to\infty$, we instead have a
    factor $e^{-n}$, which decays exponentially, like the full-plane two-point function
    in the massive regime.
\end{remark}

\section{Proof of the multipoint Ginibre bound (Equation~\eqref{eq:multi_point_ginibre})}
\label{sec:the_technical_step}

As announced before, the purpose of this section is to find a universal constant $c>0$ such that
Equation~\eqref{eq:multi_point_ginibre} holds true with $f(n,s)=c^{n^2s^2}$.

\subsection{Proof overview}

The proof has similarities with Section~\ref{sec:reverseinequality}.
We define an analogue $\Xi'$ of $\Xi$ but, for the boundary conditions corresponding to
the slope $s$,
rather than the boundary condition $01$ of Section~\ref{sec:reverseinequality}.
In the Kadanoff--Ceva representation, $\Xi'(\Gamma)$ is a multi-point correlation function,
which means that the Ginibre inequality does not apply directly.
However, the Ginibre inequality \emph{does} apply to $\Xi'(\Gamma)$
when $\Gamma$ can be decomposed into connected components,
such that each component contains at most one source and one sink of the correlation function,
so that the multi-point correlation function factorises into a product of two-point functions.

Subsection~\ref{subsec:adapting_notation} adapts the notations
from Section~\ref{sec:reverseinequality} to the case of multiple boundary jumps, and
defines the function $\Xi'$.
Subsection~\ref{subsec:sensitivity_bound} establishes a sensitivity bound for $\Xi'$,
which is a version of Lemma~\ref{lem:exploration_bound_ratio_bound} adapted to the case
of multiple boundary jumps. This sensitivity bound allows us to compare $\Xi'(\Gamma)$ to
the product of $\Xi'$ on smaller domains, each containing at most one boundary jump, and
thus to apply the Ginibre inequality on those smaller domains.
The ratio is expressed in terms of the probability of a certain percolation event, and we
control this probability in Subsection~\ref{subsec:perco}.
Finally, Subsection~\ref{subsec:the_end} puts everything together to obtain the desired inequality.

Throughout this entire section, we fix the values of
$n$ (large), $\beta\in[0,\infty)$, and $s\in(0,1/1000)$.

\subsection{Random boundary conditions and the function $\Xi'$}
\label{subsec:adapting_notation}

Enumerate the elements of the jump positions $I_s\cap[-n,n]$ as $(x_i)_{0\leq i\leq N}$
for some $N\geq 0$.
By analogy with Equation~\eqref{eq:cable_property},
we consider domains $\Gamma$ satisfying
\begin{equation}\label{eq:cable_prop_1}
    \Gamma\subset\BC\cap(\R\times(-n-\tfrac12,n+\tfrac12)),
\end{equation}
and, for each $i$,
\begin{equation}\label{eq:cable_prop_2}
    \BC\cap([x_i-1,x_i+1]\times(-n-\tfrac12,n+\tfrac12))\subset
    \Gamma.
\end{equation}

Following Section~\ref{sec:reverseinequality}, we define a random
boundary condition measure on these domains. We write ``$0$'' for
the boundary height function that is identically zero on $\partial\Gamma$, and
``$s$'' for the boundary height function corresponding to a slope $s$. We then
define the probability measure
\begin{equation}
    \nu_{\Gamma}':=\BM_{\Gamma,\beta}[\blank|\{h|_{\partial\Gamma}\in\{0,s\}\}],
\end{equation}
and we write $b$ for the $\{0,s\}$-valued random variable that encodes the new
height of the boundary (notice that the boundary condition $s$ applies
makes sense on \emph{any} continuum domain).
A \emph{good} domain $\Gamma$ is a domain
which satisfies Equation~\ref{eq:cable_prop_1} and Equation~\eqref{eq:cable_prop_2}, and
whose connected components are simply connected.
For any good domain $\Gamma$, we define
\begin{equation}\label{eq:multiple_point_cable_system}\Xi'(\Gamma)\coloneqq\Big\langle
    \Big[
        \prod\nolimits_{u\in T_n(s)} \bar\sigma_u
    \Big]\Big[
        \prod\nolimits_{v\in B_n(s)}\sigma_v
    \Big]
    \Big\rangle_{\Gamma^*,J(\Gamma)\beta}=
    \frac{
        \BM_{\Gamma,\beta}[\{h|_{\partial\Gamma}=s\}]
    }{
        \BM_{\Gamma,\beta}[\{h|_{\partial\Gamma}=0\}]
    }=\frac{\nu'_{\Gamma}[\{b=s\}]}{\nu'_{\Gamma}[\{b=0\}]}.
\end{equation}
Here we write $\langle\blank\rangle_{\Gamma^*,J(\Gamma)\beta}$ for the product
of $\langle\blank\rangle_{Q^*,J(Q)\beta}$ over the connected components $Q$ of $\Gamma$,
and we used the Kadanoff--Ceva correspondence (Lemma~\ref{lem:kadanoff_ceva_specific}) to
each such connected component.

\begin{remark}
    \label{rem:ginibre_applies}
    We stress that unlike $\Xi$, the function $\Xi'$ is not known to be
    increasing. However, if $\Gamma$ contains $|T_n(s)|$ connected components
    such that each of those components contains exactly one source and one sink of the
    correlation function, then $\Xi'(\Gamma)$ factorises into a product of two-point
    functions, and the Ginibre inequality applies to each of those two-point functions.
    This means that $\Xi'$ is increasing over the set of domains $\Gamma$ satisfying this
    additional decomposition requirement.
\end{remark}

\subsection{Another sensitivity bound}
\label{subsec:sensitivity_bound}

\newcommand\LINE{\mathfrak{L}}
\newcommand\STRIP{\calS}

Recall that the jump positions are enumerated via $I_s\cap[-n,n]=\{x_0,\dots,x_{N}\}$ for
some $N\geq 0$.
Write $y_i:=\frac{x_i+x_{i+1}}{2}$ for $0\leq i\leq N-1$.
For any $a\in\R$,
introduce the \emph{line} $\LINE(a):=\{a\}\times\R$
and the \emph{strip} $\STRIP(a):=[a-\frac{1}{20s},a+\frac{1}{20s}]\times\R$.
Write
\begin{equation}
    \STRIP_y := \cup_i \STRIP(y_i);
    \qquad
    \LINE_y := \cup_i \LINE(y_i).
\end{equation}

\begin{lemma}\label{lem:reverse_exploration_bound}
    For any good domain $\Gamma$, we have
    \begin{equation}\label{eq:reverse_exploration_bound}
        \frac{\Xi'(\Gamma)}{\Xi'(\Gamma\setminus
        \STRIP_y)}\geq \tilde\mu^0_{\Gamma,\beta}[\{
        \calE_0(\LINE_y)\subset \STRIP_y\}].
    \end{equation}
    Recall that $\calE_0(\blank)$ denotes the union of the connected components of
    $\{h=0\}$ intersecting the given set.
\end{lemma}

\begin{proof}
    First observe that
    \begin{multline}
        \label{eq:reverse_exploration_first_step}
        \frac{\Xi'(\Gamma)}{\tilde\mu^0_{\Gamma,\beta}[\{
        \calE_0(\LINE_y)\subset \STRIP_y\}]}
        =
        \frac{\nu'_\Gamma[\{b=s\}]}{\nu'_\Gamma[\{b=0\}]}
        \cdot
        \frac1{ \tilde\mu^0_{\Gamma,\beta}[\{
        \calE_0(\LINE_y)\subset \STRIP_y\}]}
        \\
        =
        \frac{\nu'_\Gamma[\{b=s\}]}{\nu'_\Gamma[\{b=0\}\cap \{
        \calE_0(\LINE_y)\subset \STRIP_y\}]}
        \geq
        \frac{\nu'_\Gamma[\{b=s\}\cap(\bigcap_{i=1}^{N-1}
        \{\calE_{a_i}(\LINE(y_i))\subset\STRIP(y_i)\})]}{\nu'_\Gamma[\{b=0\}\cap \{
        \calE_0(\LINE_y)\subset \STRIP_y\}]},
    \end{multline}
    where $a_i$ is the value of the boundary condition $s$ between the two consecutive
    jumps $x_i$ and $x_{i+1}$.
    It suffices to prove that this last expression is at least $\Xi'(\Gamma\setminus\STRIP_y)$.

    Now write $a_i(b):=a_i\cdot \true{b=s}$.
    This is the precisely the boundary condition between the jumps
    $x_i$ and $x_{i+1}$, for the boundary height function $b\in\{0,s\}$.
    The last ratio in Equation~\eqref{eq:reverse_exploration_first_step}
    is now equal to
    \begin{equation}
        \frac{
            \nu'_\Gamma\Big[\{b=s\}\Big|\bigcap_{i=1}^{N-1}
            \{\calE_{a_i(b)}(\LINE(y_i))\subset\STRIP(y_i)\}\Big]
        }{
            \nu'_\Gamma\Big[\{b=0\}\Big|\bigcap_{i=1}^{N-1}
            \{\calE_{a_i(b)}(\LINE(y_i))\subset\STRIP(y_i)\}\Big]
        }=:R.
    \end{equation}
    Again, it suffices to prove that this is at least $\Xi'(\Gamma\setminus\STRIP_y)$.

    If we condition on the precise shape $\{\cup_i \calE_{a_i(b)}(\LINE(y_i))=\calA\}$,
    then it is easy to calculate the ratio of the conditional probabilities;
    we have
    \begin{equation}
        \frac{
            \nu'_\Gamma\Big[\{b=s\}\Big|\{\cup_i \calE_{a_i(b)}(\LINE(y_i))=\calA\}\Big]
        }{
            \nu'_\Gamma\Big[\{b=0\}\Big|\{\cup_i \calE_{a_i(b)}(\LINE(y_i))=\calA\}\Big]
        }
        =
        \Xi'(\Gamma\setminus\calA).
    \end{equation}
    Moreover, since any such exploration disconnects the domain into components
    containing at most one source and one sink, the Ginibre inequality applies to each of
    those components (Remark~\ref{rem:ginibre_applies}).
    In particular, if $\calA\subset\STRIP_y$,
    then this ratio is at least $\Xi'(\Gamma\setminus\STRIP_y)$.
    We now obtain the desired inequality $R\geq \Xi'(\Gamma\setminus\STRIP_y)$
    by averaging over all possible shapes $\calA$.
\end{proof}

\subsection{Percolation with boundary zero boundary height}
\label{subsec:perco}
Next, we study the right-hand side of Lemma~\ref{lem:reverse_exploration_bound}.

\begin{lemma}\label{lem:path_existence}
    The right-hand side of Equation~\eqref{eq:reverse_exploration_bound}
    is at least $(\cfav)^{1000 (ns)^2}$,
    where $\cfav\in(0,1)$ is some universal constant.
\end{lemma}

\begin{proof}
    It was proved in~\cite[Lemma~7.2]{Lammers_2023_DichotomyTheoryHeight} that there
    exists a universal constant $\cfav\in(0,1)$ such that
    \begin{equation}
        \tilde\mu^0_{\Gamma,\beta}[\{
        \calE_0(\LINE(y_i))\subset \STRIP(y_i)\}]
        \geq (\cfav)^{100ns}.
    \end{equation}
    The percolation $\calE_0(\LINE_y)$ is a decreasing function
    of the absolute-value height function $|h|$.
    This percolation satisfies the FKG inequality~\cite[Lemma~2.15 and
    Section~12]{Lammers_2023_DichotomyTheoryHeight}.
    This yields
    \begin{equation}
        \tilde\mu^0_{\Gamma,\beta}[\cap_i\{
        \calE_0(\LINE(y_i))\subset \STRIP(y_i)\}]
        \geq (\cfav)^{1000(ns)^2}.
    \end{equation}
    This is the desired inequality.
\end{proof}

\subsection{Combining the different estimates}
\label{subsec:the_end}

\begin{proof}[Proof of Equation~\ref{eq:multi_point_ginibre} with $f(n,s)=c^{n^2s^2}$]
    Applying Lemma~\ref{lem:reverse_exploration_bound} with $\Gamma=\Gamma_n$ and then
    Lemma~\ref{lem:path_existence} gives
    \begin{equation}
        \Xi'(\Gamma_n)
        \geq
        (\cfav)^{1000(ns)^2}\Xi'(\Gamma_n\setminus\STRIP_y).
    \end{equation}
    By the definition of $\Xi'$ and the Kadanoff--Ceva correspondence,
    \begin{equation}
        \Xi'(\Gamma_n)
        =
        \Big\langle\Big[\prod\nolimits_{u\in
        T_n(s)}\bar\sigma_u\Big]\Big[\prod\nolimits_{v\in
        B_n(s)}\sigma_v\Big]\Big\rangle_{\Lambda_n,\beta},
    \end{equation}
    while the cable domain associated with $\Lambda'_n(s)$ is contained in
    $\Gamma_n\setminus\STRIP_y$. The latter domain is decomposed into components
    containing at most one source and one sink, so Remark~\ref{rem:ginibre_applies}
    and the Ginibre inequality imply
    \begin{equation}
        \Xi'(\Gamma_n\setminus\STRIP_y)
        \geq
        \Big\langle\Big[\prod\nolimits_{u\in
        T_n(s)}\bar\sigma_u\Big]\Big[\prod\nolimits_{v\in
        B_n(s)}\sigma_v\Big]\Big\rangle_{\Lambda'_n(s),\beta}.
    \end{equation}
    Setting $c:=\cfav^{1000}$ and combining the above gives
    Equation~\eqref{eq:multi_point_ginibre} with $f(n,s)=c^{n^2s^2}$.
\end{proof}

\subsection*{Acknowledgements}
\addcontentsline{toc}{section}{Acknowledgements}
The authors thank Thierry Bodineau for several discussions regarding the ideas in this article.
This research was supported by the French National Research Agency
(ANR), project number ANR-23-CPJ1-0150-01.

\printbibliography
\end{document}